%% file: Main_revision1.tex
\documentclass[11pt,a4paper]{article}
\usepackage[top=1in, bottom=1.25in, left=1in, right=1in]{geometry}
\usepackage[utf8]{inputenc}
\usepackage[english]{babel}
\usepackage{graphicx,epstopdf}
\usepackage{caption, subcaption,float}
\usepackage{color}
\usepackage{multirow}
\usepackage{tikz}
\usepackage{amsthm}
\usepackage{tabls}
\usepackage{cite}
\usepackage[round,comma,sort]{natbib}

\usepackage[ruled, linesnumberedhidden,vlined]{algorithm2e} 
\usetikzlibrary{matrix}
\usepackage{amsmath,amssymb,dsfont,mathtools}
\usepackage{hyperref}
\hypersetup{
    colorlinks=true,       
    linkcolor=blue,          
    citecolor=olive,       
    filecolor=magenta,      
    urlcolor=blue          
}
\usepackage{aliascnt}

\usepackage{soul}

\begin{document}

\newcommand{\ca}{\mathcal{C}_{AL}}

\title{The conjugacy stability problem \\ for parabolic subgroups in Artin groups}
\date{\today }
\author{Mar\'{i}a Cumplido}


\maketitle
\theoremstyle{plain}
\newtheorem{theorem}{Theorem}

\newaliascnt{lemma}{theorem}
\newtheorem{lemma}[lemma]{Lemma}
\aliascntresetthe{lemma}
\providecommand*{\lemmaautorefname}{Lemma}

\newaliascnt{proposition}{theorem}
\newtheorem{proposition}[proposition]{Proposition}
\aliascntresetthe{proposition}
\providecommand*{\propositionautorefname}{Proposition}

\newaliascnt{corollary}{theorem}
\newtheorem{corollary}[corollary]{Corollary}
\aliascntresetthe{corollary}
\providecommand*{\corollaryautorefname}{Corollary}

\newaliascnt{conjecture}{theorem}
\newtheorem{conjecture}[conjecture]{Conjecture}
\aliascntresetthe{conjecture}
\providecommand*{\conjectureautorefname}{Conjecture}

\theoremstyle{remark}

\newaliascnt{claim}{theorem}
\newaliascnt{remark}{theorem}
\newtheorem{claim}[claim]{Claim}
\newtheorem{remark}[remark]{Remark}
\newaliascnt{notation}{theorem}
\newtheorem{notation}[notation]{Notation}
\aliascntresetthe{notation}
\providecommand*{\notationautorefname}{Notation}

\aliascntresetthe{claim}
\providecommand*{\claimautorefname}{Claim}

\aliascntresetthe{remark}
\providecommand*{\remarkautorefname}{Remark}

\newtheorem*{claim*}{Claim}
\theoremstyle{definition}

\newaliascnt{definition}{theorem}
\newtheorem{definition}[definition]{Definition}
\aliascntresetthe{definition}
\providecommand*{\definitionautorefname}{Definition}

\newaliascnt{example}{theorem}
\newtheorem{example}[example]{Example}
\aliascntresetthe{example}
\providecommand*{\exampleautorefname}{Example}

\def\autorefspace{\hspace*{-0.5pt}}
\def\sectionautorefname{Section\autorefspace}
\def\subsectionautorefname{Section\autorefspace}
\def\subsubsectionautorefname{Section\autorefspace}
\def\figureautorefname{Figure\autorefspace}
\def\subfigureautorefname{Figure\autorefspace}
\def\tableautorefname{Table\autorefspace}
\def\equationautorefname{Equation\autorefspace}
\def\Itemautorefname{item\autorefspace}
\def\Hfootnoteautorefname{footnote\autorefspace}
\def\AMSautorefname{Equation\autorefspace}

\newcommand{\co}{\simeq_c}
\newcommand{\w}{\widetilde}
\newcommand{\po}{\preccurlyeq}
\newcommand{\dist}{\mathrm{d}}

\def\Z{\mathbb Z} 
\def\Ker{{\rm Ker}} \def\R{\mathbb R} \def\GL{{\rm GL}}
\def\HH{\mathcal H} \def\C{\mathbb C} \def\P{\mathbb P}
\def\SSS{\mathfrak S} \def\BB{\mathcal B} 
\def\supp{{\rm supp}} \def\Id{{\rm Id}} \def\Im{{\rm Im}}
\def\MM{\mathcal M} \def\S{\mathbb S}
\newcommand{\bigveer}{\bigvee^\Lsh}
\newcommand{\wedger}{\wedge^\Lsh}
\newcommand{\veer}{\vee^\Lsh}
\def\diam{{\rm diam}}

\newcommand{\myref}[2]{\hyperref[#1]{#2~\ref*{#1}}}

\begin{abstract}

Given an Artin group $A$ and a parabolic subgroup $P$, we study if every two elements of $P$ that are conjugate in $A$, are also conjugate in $P$. We provide an algorithm to solve this decision problem if $A$ satisfies three properties that are conjectured to be true for every Artin group. This allows to solve the problem for new families of Artin groups. We also partially solve the problem if $A$ has $FC$-type, and we totally solve it if $A$ is isomorphic to a free product of Artin groups of spherical type. In particular, we show that in this latter case, every element of $A$ is contained in a unique minimal (by inclusion) parabolic subgroup. 

\medskip

{\footnotesize

}

\end{abstract}

\section{Introduction}

Artin (or Artin--Tits) groups were defined by Jacques Tits in the 60's. They are groups presented by a finite set of generators $S$ and at most one relation of the form $stst\cdots = tsts\cdots$, for every pair $s,t\in S$, with the same number of letters $m_{s,t}$ at each side of the equality. If there is no relation associated to a pair of generators $s,t\in S$, then we denote $m_{s,t}=\infty$. Then, the presentation of an Artin group is as follows:
 
$$A_S=\langle S \,|\, \underbrace{sts\dots}_{m_{s,t} \text{ elements}}=\underbrace{tst\dots}_{m_{s,t} \text{ elements}} \forall s,t\in S,\, s\neq t,\, m_{s,t}\neq \infty \rangle.$$

These groups are algebraic generalisations of the well-known braid groups on $n+1$ strands \citep{Artin}:

\begin{equation*}
A_n=\left\langle \sigma_1,\dots, \sigma_{n}\, \begin{array}{|lr}
                                              \sigma_i\sigma_j=\sigma_j\sigma_i, & |i-j|>1 \\
                                                \sigma_i\sigma_{j}\sigma_i= \sigma_{j}\sigma_{i}\sigma_{j} , & |i-j|=1
                                             \end{array}
 \right\rangle .\end{equation*}
A fundamental tool for the study of braid groups is the action by isometries of $A_n$  on the curve complex of the $n+1$-punctured disk $\mathcal{D}_{n+1}$. The curve complex has as vertices (isotopy classes of non-degenerated) simple closed curves in $\mathcal{D}_{n+1}$. For Artin groups, the analogous of simple closed curves are irreducible parabolic subgroups. In fact, there is a bijection between the proper irreducible parabolic subgroups of $A_n$ and the simple closed curves in $\mathcal{D}_{n+1}$ (see an explanation in \citealp[Section~2]{CGGW}).

\medskip

A \emph{standard parabolic subgroup} $A_X$ is a subgroup generated by a subset of generators $X\subseteq S$. The conjugate of any standard parabolic subgroup by an element of $A_S$ is called a \emph{parabolic subgroup}.
The study of parabolic subgroups has been an important source of research in Artin groups over the last forty years. These are natural and easy to define subgroups. They are the main ingredient of complexes in which Artin groups act, as the Deligne complex \citep{Deligne,CharneyDavis} or the complex of irreducible parabolic subgroups \citep*{CGGW}. However, as it happens for most questions in Artin groups, basic properties of parabolic subgroups are in general unknown. Some of the facts we know about are the following: In his thesis, \cite{Vanderlek} proved that a standard parabolic subgroup is again an Artin group, and we also know that they are convex in every case \citep{CP}. The structure of centralisers of parabolic subgroups and many of their properties have been well studied only certain cases by \cite{Paris} and \cite{Godelle,Godelle2,GodelleFC}, among others; and we only know if the intersection of parabolic subgroups is again a parabolic subgroup for a few families of Artin groups \citep{CGGW, Rose, CMV}. 

\medskip

In this paper we discuss in which cases embeddings  of parabolic subgroups into the Artin group merge conjugacy classes. This is also called the \emph{conjugacy stability} problem for parabolic subgroups.

\begin{definition}
A parabolic subgroup $P$ of an Artin group $A$ is \emph{conjugacy stable} in $A$ if for every $x,y\in P$ such that $g^{-1}xg=y$, $g\in G$, there is $\hat{g}\in P$ such that $\hat{g}^{-1}x\hat{g}=y$.
If $P$ is not conjugacy stable in~$A$ we say that the inclusion of $P$ into $A$ \emph{merges conjugacy classes}.
\end{definition}
 
This problem has been solved only for some specific families of Artin groups. \cite{Meneses} proved that parabolic subgroups of braid groups are always conjugacy stable. However, for Artin groups this is not always the case. In \citep*{CCC}, we give an explicit classification for \emph{spherical-type} (or finite type) Artin groups, which are the groups that become finite when adding to their presentation the relations $s^2=1$ for every $s\in S$. For large type and FC-type, a simpler question was addressed by \cite{GodelleConjugacy}: He studied what happens  if in the definition of conjugacy stable we impose $g$ to be an element of $S$. 
At the end of \citep*{CMV}, we completely classify the parabolic subgroups of large Artin groups up to conjugacy stability, using the aforementioned results of Paris and Godelle. The aim of this article is to use these results to prove that conjugacy stability problem can be solved for every Artin group satisfying three properties that are conjectured to always hold in Artin groups.  

\smallskip

If for an element~$\alpha$ in an Artin group there is a unique minimal (with respect to the inclusion) parabolic subgroup~$P_\alpha$ containing $\alpha$, we say that $P_\alpha$ is the \emph{parabolic closure} of $\alpha$. We will show:

\bigskip
\noindent
\textbf{Theorem A.} \textit{Let $A$ be a standardisable (\autoref{def:standardizable}) Artin group satisfying the ribbon property (\autoref{def:ribbon}) and such that every element in $A$ has a parabolic closure. Then, there is an algorithm that decides whether a parabolic subgroup~$P$ of $A$ is conjugacy stable in~$A$ or not.}

\bigskip
\noindent  The existence of parabolic closures --which is a consequence of the intersection of two parabolic subgroups being a parabolic subgroup-- and the other two hypotheses of the theorem are conjectured to be true for all Artin groups. In particular, they are known to be true for spherical-type Artin groups \citep{Godelle,CGGW}. 
The standardisation and ribbon properties are true for FC-type and two-dimensional Artin groups \citep{Godelle2,GodelleFC}.
 For FC-type, the problem of the intersection of parabolic subgroups is solved for spherical-type parabolic subgroups --the conjugates of some spherical-type standard parabolic subgroup-- by \cite{Rose}. Using these results and the fact that FC-type Artin groups can be seen as amalgamated free-products of spherical-type Artin groups, we will partially solve the conjugacy stability problem for parabolic subgroups of a FC-type Artin group $A$. We will totally solve the problem if $A$ is isomorphic to a free-product of spherical-type Artin groups, by proving the existence of parabolic closures in this case (\autoref{Prop:clausura_free}). This is summarized in Theorem~B. 

\begin{definition}
Given an Artin group $A$ and a parabolic subgroup $P$ of $A$, we say that $P$ is \emph{conjugacy quasi-stable} if for every two elements $x,y\in P$ contained in (possibly different) spherical-type parabolic subgroups of $A$ such that $g^{-1}xg=y$ with $g\in A$, there is $z\in P$ such that $z^{-1}xz=y$.
\end{definition} 
 
 \begin{remark}
Notice that for spherical-type Artin groups being conjugacy quasi-stable is equivalent to be conjugacy stable.
\end{remark}

\noindent 
\textbf{Theorem B.} \textit{Let $A$ be an FC-type Artin group. There is an algorithm that decides whether a given parabolic subgroup~$P$ of $A$ is conjugacy quasi-stable in $A$. In particular, this algorithm can tell whether a spherical-type parabolic subgroup is conjugacy stable or not. }

\textit{
Moreover, if $A$ is isomorphic to a free product of spherical type Artin groups, then every element of~$A$ has a parabolic closure and there is an algorithm that solves the conjugacy stability problem for every parabolic subgroup of $A$.
}

\bigskip

This article is structured in the following way: In \autoref{Section2} we will describe a result of \cite{Paris} that gives an algorithm to decide when two standard parabolic subgroups are conjugate in any Artin group, and we will give  an explicit form for this algorithm; in \autoref{Section3}  we will explain how to modify this algorithm to solve the conjugacy stability problem for parabolic subgroups of Artin groups that satisfy the three hypothesis of Theorem~A; in \autoref{Section4} we will discuss the case of FC-type Artin groups. 

\begin{remark}
After a first preprint of this paper, \cite{MartinB} generalised the results in \citep*{CMV} and showed that the intersection of parabolic subgroups is a parabolic subgroup for two-dimensional Artin groups with a Coxeter graph --see next section-- in which every vertex is disconnected from at most one other vertex. This completed the set of three hypotheses needed in Theorem~A and allowed him two apply \myref{Algoritmo:main}{Algorithm} of \autoref{Section3} to solve the conjugacy problem in this case.

\cite{Haettel} has also proved the three conjectures for Euclidean Artin groups of type $\tilde{A}$ and $\tilde{C}$, so we know that the main theorem works for these groups.

\end{remark}

\section{Conjugate standard parabolic subgroups}\label{Section2}

In this section we explain in detail the results in \citep{Paris} to decide when two standard parabolic subgroups are conjugate in an Artin group $A_S$. This work is based on the part of Daan Krammer's thesis that solves the conjugacy problem in Coxeter groups, which is published in \citep{Krammer}. To begin, we first need to know how to define the Coxeter graph of an Artin group and the classification of Artin groups of spherical type.

\begin{definition}
The \emph{Coxeter graph} $\Gamma_S$ of the Artin group $A_S$ is the graph defined by the following data:

\begin{itemize}

\item The set of vertices of $\Gamma_S$ is $S$.

\item There is an edge connecting $s$ and $t$ if and only if $m_{s,t}>2$. This edge if labeled with $m_{s,t}$ if $m_{s,t}>3$.

\end{itemize}
\noindent
If $\Gamma_S$ is connected, we say that $A_S$ is \emph{irreducible}.

\end{definition}

\noindent
In \autoref{coxeter}, the reader can find the classification \citep{Coxeter} of the ten types of irreducible Artin groups of spherical type. All the other Artin groups of spherical type are direct products of irreducible ones. When useful, we will refer to $A_S$ as $A_n$, $B_n$, $D_n \dots$, but normally we will say that the Artin group and Coxeter graph are of type $A_n$, $B_n$, $D_n \dots$ We denote the generators of $A_S$ by $s_1,s_2,s_3,\dots$, accordingly with the numbering of \autoref{coxeter}.

\begin{figure}[h]
  \centering
  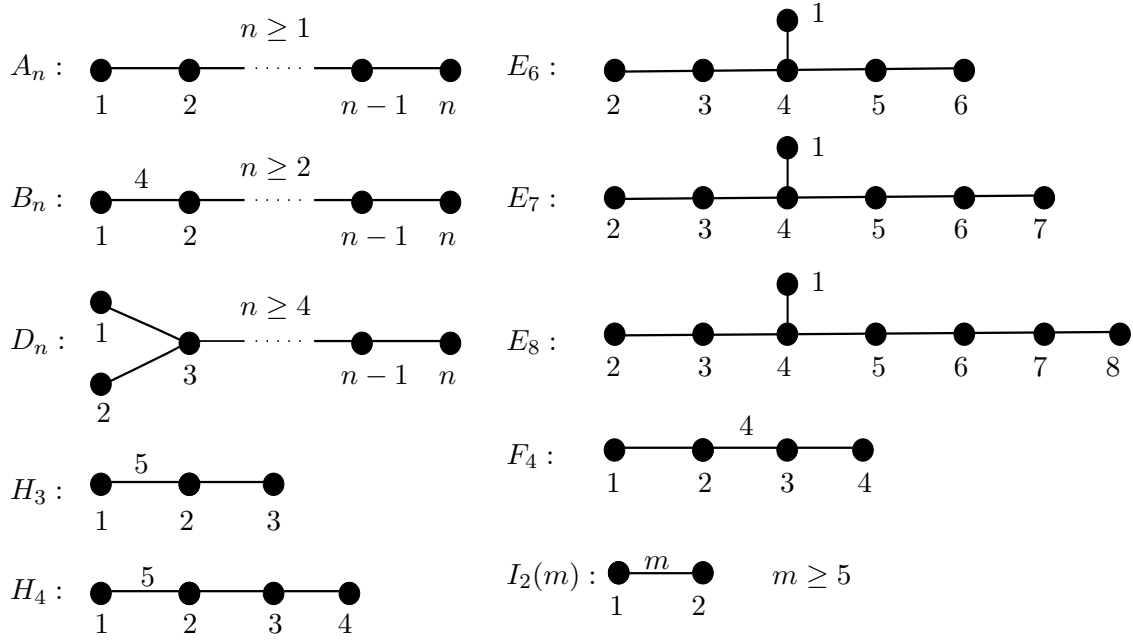
  \medskip
  \caption{Classification of irreducible Coxeter graphs of finite type}
  \label{coxeter}
\end{figure}

\smallskip

 Given an Artin group $A_S$, the submonoid $A_S^+$ of $A_S$ generated by $S$ has the exactly same presentation as $A_S$ (seen as monoid) \citep{Paris2}. If $A_S$ is an Artin group of spherical type, it has a Garside structure. This implies that if $A_S$ has spherical type there is a lattice order $\po$ defined by ``$a\po b$ iff $\exists c\in A_S^+,$ $ac=b$''. The least common multiple of all generators of $S$ is called the Garside element of $A_S$ and is denoted by $\Delta$. By \citep{BS} we know that the centre $Z(A_S)$ of $A_S$ is generated either by $\Delta$ or by $\Delta^2$. 

\smallskip

For Artin groups of type $A_n\, (n \geq 2), D_n \,(n \geq 5), E_6$ and $I_2(m) \,(m\geq 5 \text{ and odd} )$, $\Delta^2$ generates the centre of the group. Otherwise, $\Delta$ generates the centre of $A_S$. In the first case, the conjugation by~$\Delta$ can be seen as a reflection automorphism  of~$\Gamma_S$.  These conjugations, that are well-known by experts, are detailed in what follows:

\begin{itemize}

\item For $A_n,\, n\geq 2$, one has that $\Delta^{-1}s_i\Delta = s_{n-i+1}$.

\item For $D_n,$ with $n \geq 5$ and $n \text{ odd}$, the conjugation by $\Delta$ permutes $s_1$ and $s_2$ and fixes the other generator.

\item For $E_6$, the conjugation by $\Delta$ fixes $s_1$ and $\Delta^{-1}s_i\Delta = s_{8-i}$ for $i\neq 1$.

\item For $I_2(m),$ with $m \geq 5$  and  $m \text{ odd}$, the conjugation by $\Delta$ permutes $s_1$ and $s_2$.

\end{itemize}

We will be specially interested in the Artin groups such that the conjugation by~$\Delta$ can be seen as a reflection automorphism  of~$\Gamma_S$:

\begin{definition}
We say that $A_S$ is a \emph{twistable} Artin group if it is one of the following Artin groups of spherical type:
\[A_n,\, n\geq 2; \quad D_n,\, n \geq 5 \text{ and } n \text{ odd;}\quad E_6; \quad I_2(m),\,  m \geq 5 \text{ and } m \text{ odd.}\]
\end{definition}

\medskip

Thanks to \cite{Vanderlek}, we also know that a standard parabolic subgroup $A_Y$ is an Artin group having as Coxeter graph $\Gamma_Y\subset \Gamma_S$. If~$A_Y$ has spherical type, we denote its Garside element by $\Delta_Y$. 

\smallskip

Suppose that~$A_X$ is a maximal proper standard parabolic subgroup of a twistable standard parabolic subgroup~$A_Y$ of~$A_S$, in other words, $X=Y\setminus\{t\}$, $t\in X$. If~$A_Y$ is of type~$A_n$, $n$~odd, suppose that~$t$ is not the central generator of~$A_Y$. If~$A_Y$ is of type~$E_6$ suppose that~$t$ does not correspond to~$s_1$ or~$s_4$ and if $A_Y$ is of type $D_n$ suppose that $t$ corresponds to either~$s_1$ or~$s_2$. Then $\Delta_Y^{-1}A_X\Delta_Y$ is a standard parabolic subgroup of~$A_Y$ different from~$A_X$. (If $t$ is one of the forbidden generators, then $\Delta_Y^{-1}A_X\Delta_Y=A_X$). This is the main ingredient of Paris' result, which states that two standard parabolic subgroups~$A_X$ and~$A_{X'}$ are conjugate if and only if it is possible to go from one to the other by performing those types of conjugations or ``twists''.


\medskip

Let $X\subset S$ and define $\mathrm{Adj}(X)$ as the set of vertices in $\Gamma_S$ that are adjacent to $\Gamma_X$. We will consider lists of couples $(Y,c)$, where  $ Y \subset S $ is a subset of generators and  $c\in A_S$ is an element that conjugates the set $X$ to the set $Y$.
For a given $X\subset S$, we will recursively construct the list $\mathcal{V}_X$ as follows. Start the list with the couple $(X,1)$. For every $(Y,c)$ in the list and for every $t\in \mathrm{Adj}(Y)$, take the connected component $\Gamma_{Y'}$ of $\Gamma_{Y\cup \{t\}}$ containing $t$. If this component is twistable, conjugate $Y$ by the Garside element $\Delta_{Y'}$ of the component. If the result $Z$ is a subset of generators that is not contained in some couple of the list, add the couple $(Z,c \Delta_{Y'})$. Repeat this process. To prove that the process stops at some point, just observe that the set of standard parabolic subgroups of an Artin group is finite.

\begin{theorem}\label{Paris}
Given an Artin group $A_S$ and two standard parabolic subgroups~$A_X$ and~$A_{X'}$, $A_X$ is conjugate to~$A_{X'}$ if and only if there is a couple $(X',c)$ in $\mathcal{V}_X$, in which case~$c$ is a conjugacy element. 
\end{theorem}

\begin{proof}
This theorem is a reformulation of \cite[Theorem~4.1]{Paris}. We can see that $c$ is a conjugacy element by its own construction. 
\end{proof}

In \myref{Algoritmo:Paris}{Algorithm}, we give to Paris' result an explicit algorithmic form. The algorithm tells us when two standard parabolic subgroups $A_X$ and $A_{X'}$ are conjugate. If they are not, it constructs the whole list $\mathcal{V}_X$. 

\bigskip

\begin{algorithm}[H]
\SetKwInOut{Input}{Input}\SetKwInOut{Output}{Output}

\BlankLine

\Input{The Coxeter graph $\Gamma_S$ of an Artin group $A_S$ and two subsets $X,{X'}\subset S$.}
\Output{A conjugating element between the parabolic subgroups $A_X$ and $A_{X'}$ or ``There is no conjugating element''.}
\BlankLine
\BlankLine

\If{$|X|\neq|{X'}|$ }{ \smallskip
\Return{``There is no conjugating element'';}}

\medskip

$\mathcal{V}=\{(X,1)\}$;


\smallskip

\For{$(Y,c) \in \mathcal{V}$}{

\smallskip

\For{$t\in \mathrm{Adj}(Y)$}{
\smallskip

\If{the connected component $\Gamma_{Y'}$ of $\Gamma_{Y\cup \{t\}}$ containing $t$ is twistable}{

\smallskip
$Z= \Delta_{Y'}^{-1} Y \Delta_{Y'}$;

\smallskip
\If{ $Z$ is not the first element of any couple in $\mathcal{V}$}
{
\smallskip

$\mathcal{V}=\mathcal{V}\cup \{(Z,c \Delta_{Y'})\}$;

}

\smallskip
\If{$Z={X'}$}
{

\smallskip

\Return{$c \Delta_{Y'}$ ;}

}

}}}

 \Return{``There is no conjugating element'';}
\caption{Algorithm that finds a conjugating element between two standard parabolic subgroups or tells that it does not exist.}
\label{Algoritmo:Paris}
\end{algorithm}

\bigskip

\noindent\textbf{Example.}\,  Consider the spherical-type Artin group~$E_7$, as depicted in \autoref{coxeter}. We are going to see that the parabolic subgroup $A_X$ with $X=\{s_1,s_2,s_3,s_4,s_6\}$ is conjugate to $A_{X'}$, where $X'=\{s_2,s_4,s_5,s_6,s_7\}$. First, we take $s_5\in \mathrm{Adj}(X)$. The set of generators $X\cup \{s_5\}=\{s_1,s_2,s_3,s_4,s_5, s_6\}$ defines a connected spherical-type parabolic subgroup isomorphic to $E_6$, which is twistable. If we conjugate $X$ by the Garside element of $A_{X\cup \{s_5\}}$, we obtain the set of generators $Y=\{s_1,s_2,s_4,s_5,s_6\}$. Now take $s_7\in \mathrm{Adj}(Y)$. The group defined by $Y\cup \{s_7\}$ has the connected component $A_Z$, $Z=\{s_1,s_4,s_5,s_6,s_7\}$, which is a (twistable) braid group. Conjugating by the corresponding Garside element, we finally obtain $\Delta_Z^{-1}Y\Delta_Z = X'$.

\section{Solution to the conjugacy stability problem}\label{Section3}

In this section we will explain two of the three hypothesis of Theorem~A, namely the ribbon property and the property of being standardisable. After that, we will construct the main algorithm of this paper to know when the embedding of a standard parabolic subgroup merges conjugacy classes.

\medskip

We first describe the results of \citep{Godelle, Godelle2, GodelleFC} concerning the set of elements conjugating two standard parabolic subgroup of an Artin group $A_S$. Suppose that $A_{X}$, $X\subset S\setminus T$, is a standard parabolic subgroup of spherical type and let $X'=X\setminus \{t\}$, for some $t\in X$. Since $X$ has spherical type, $X'$ also has spherical type and we can consider $\Delta_X$ and $\Delta_{X'}$.  We have that  $$\Delta_{X}^{-1}\Delta_{X'}A_{X'}\Delta_{X'}^{-1}\Delta_{X}= \Delta_{X}^{-1}A_{X'}\Delta_{X}=A_Y,$$
for some subset $Y\subset X$.  The conjugating element 
$\Delta_{X'}^{-1}\Delta_{X}$ and its inverse is what Godelle respectively calls an elementary {$(X',Y)$--ribbon} and an elementary {$(Y,X')$--ribbon}. 

In general, for any (not necessarily of spherical type) parabolic subgroup $A_T$ of $A_S$, if there is $s\in S$ such that the component $\Gamma_{U}$ of $\Gamma_{T\cup \{s\}}$ that contains~$s$ is of spherical type, we call  $r_{T,s}:=\Delta_{U\setminus\{s\}}^{-1}\Delta_{U}$ and its inverse \emph{elementary ribbons} --notice that $\Gamma_{U}$ does not need to be twistable--. We say that an element $r=r_1r_2\cdots r_q$ is a \emph{$(T,T')$--ribbon} if and only if there is a sequence of sets of generators ${T=T_1}, T_2, \dots, T_{q+1}=T'$ such that each $r_i$ is an elementary  $(T_i,T_{i+1})$--ribbon. The set of all $(T,T')$--ribbons is denoted by $\mathrm{Ribb}(T,T')$.  When referring to a $(T,T')$--ribbon without caring about the specific $T'$, we will use the term $(T,-)$--ribbon.

\medskip

Now we will see some properties about ribbons. The following lemma will allow us to work some of the proofs using positive elementary ribbons and treat the negative ones as an analogous case:

\begin{lemma}\label{lemma:negative_ribbons}
Let $A_S$ be an Artin group, $X\subset S$ and $a\in S\setminus X$. Suppose that $\Gamma_Y$ is the component of $\Gamma_{X\cup \{a\}}$ that contains $a$. If there is an elementary ribbon~$r_{X,a}= \Delta_{Y\setminus \{a\}}^{-1} \Delta_Y$, then there are $T\subset S$, $s\in S\setminus T$ such that $r_{T,s}= \Delta^{-1}_{Y\setminus \{s\}}\Delta_Y$ and $r_{X,a}^{-1}=\Delta_{Y\setminus \{s\}}\Delta_Y^{-1}$.

\end{lemma}

\begin{proof}
We have that $r^{-1}_{X,a}=\Delta_Y^{-1} \Delta_{Y\setminus \{a\}}= \Delta_{Y\setminus \{s\}}\Delta_Y^{-1}$, where $s=\Delta_Y^{-1}a\Delta_Y$. To see that there is a positive elementary ribbon of the form $\Delta_{Y\setminus \{s\}}^{-1}\Delta_Y$, let $T=(X\cup\{a\}) \setminus \{s\}$. Hence, $\Gamma_Y$ is the component of $\Gamma_{T\cup\{s\}}$ that contains $s$ and $r_{T,s}= \Delta_{Y\setminus \{s\}}^{-1}\Delta_Y$.

\end{proof}

\begin{remark}\label{remark:negativeribbons}
In the above lemma, $r_{T,s}=\Delta^{-1}_{Y\setminus \{s\}}\Delta_Y$ and $r^{-1}_{X,a}=\Delta_{Y\setminus \{s\}}\Delta_Y^{-1}$ are (positive and negative) elementary $(T,X)$--ribbons. Similarly, $r_{X,a}=\Delta_{Y\setminus \{a\}}^{-1} \Delta_Y$ and $r_{T,s}^{-1}=\Delta_{Y\setminus \{a\}} \Delta_Y^{-1}$ are (positive and negative) elementary $(X,T)$--ribbons. 
\end{remark}

The next two lemmas help us understand how the conjugation by ribbons transforms the generators of standard parabolic subgroups:

\begin{lemma}\label{lemma:elementary_ribbons}
Let $A_S$ be an Artin group and $X\subset S$. Let $t\in S\setminus X$ and $Z\subset X\cup \{t\}$ be such that~$\Gamma_Z$ is the connected component of $\Gamma_{X\cup \{t\}}$ containing $t$ and it is of spherical type. Let $X'\subseteq X$ denote any subset defining a connected component $\Gamma_{X'}$ of $\Gamma_X$.  Then

\begin{itemize}

\item If $A_{X'}$ defines a component which is not of spherical type, then $r_{X,t}^{-1}sr_{X,t}=s$, for every $s\in X'$.

\item If $A_{X'}$ is of spherical type and of type different from $A, D, E_6$ and $I_2(m)$, then $r_{X,t}^{-1}sr_{X,t}=s$, for every $s\in X'$.

\item If $A_{X'}$ is of type either $E_6$ or $I_2(m)$, then either $r_{X,t}^{-1}sr_{X,t}=s$ for every $s\in X'$ or $r_{X,t}^{-1}sr_{X,t}=\Delta_{X'}s\Delta_{X'}^{-1}$ for every $s\in X'$.

\item If $A_{X'}$ is of type either $A$ or $D$, then $r_{X,t}^{-1}X'r_{X,t}\subset X\cup\{t\}$.

\end{itemize}

\end{lemma}

\begin{proof}
If $X'$ is not of spherical type, then $X'$ cannot be contained in $Z$ and $X'$ and $Z$ are not adjacent, so the conjugation by the elementary ribbon $r_{X,t}$ does not modify $X'$. Suppose that $X'$ is different from $A$, $D$, $E_6$ and $I_2(m)$. If $X'$ is not contained in~$Z$, again~$X'$ and~$Z$ are not adjacent and therefore~$X'$ cannot be modified by a conjugation by~$r_{X,t}$. If $X'\subset Z$, then $X'\subset Z\setminus \{t\}$ so  $(X',Z)\in \{(B_{m_1}, B_{m_2}), (B_3,F_4), (H_3,H_4),(E_7, E_8)\}$, for $1<m_1<m_2$. In this case, both~$\Delta_{X'}$ and~$\Delta_Z$ are central in $A_{X'}$ and~$A_Z$ respectively. This means that $r_{X,t}^{-1}sr_{X,t}=s$, for every $s\in X'$. If $X'$ is of type~$E_6$ or~$I_2(m)$, we can suppose as before that $X'\subset Z$. In this case, $(X',Z)\in \{(E_6,E_7), (E_6,E_8) (I_2(5),H_3), (I_2(5),H_4)\}$, so $\Delta_Z$ is central in~$A_Z$ and~$\Delta_{X'}$ is not central in $A_{X'}$. Thus, $r_{X,t}^{-1}sr_{X,t}= \Delta_Z^{-1}\Delta_{X'}s\Delta_{X'}^{-1}\Delta_Z =\Delta_{X'}s\Delta_{X'}^{-1}$ for every $s\in X'$. The last item follows by definition.

\end{proof}

\begin{remark}\label{remark_negative}
By \autoref{lemma:negative_ribbons}, the previous lemma works analogously if we replace the positive elementary $(X,-)$--ribbon $r_{X,t}$ by a negative elementary $(X,-)$--ribbon.
\end{remark}

\begin{lemma}\label{lemma:permutation}
Let $A_S$ be an Artin group, $X\subset S$, and $\alpha$  be an $(X,X)$--ribbon. Let $X'\subseteq X$ denote any subset defining a connected component $\Gamma_{X'}$ of $\Gamma_X$. Then,

\begin{itemize}

\item If $A_{X'}$ has not spherical type or has a spherical type different from $A$, $D$, $E_6$ and $I_2(m)$, then $\alpha^{-1}s\alpha=s$, for every $s\in X'$.

\item If $A_{X'}$ is of type $E_6$ or $I_2(m)$, then either $\alpha^{-1}s\alpha=s$ for every $s\in X'$ or $\alpha^{-1}s\alpha=\Delta_{X'}s\Delta_{X'}^{-1}$ for every $s\in X'$.

\item If $A_{X'}$ is of type $A$ or $D$, then $\alpha^{-1}X'\alpha = X''$, where $\Gamma_{X''}$ is isomorphic to $\Gamma_{X'}$.

\end{itemize}

\end{lemma}

\begin{proof}
By definition, $\alpha$ is a product $\prod_{i=1}^k r_i$ of elementary $(X_i,Y_i)$--ribbons, $r_i$, where $Y_i=X_{i+1}$ and $X_1=Y_k=X$. When we conjugate $A_{X'}$ by an elementary  $X$--ribbon, we obtain a parabolic subgroup of the same type. Therefore, by \autoref{lemma:elementary_ribbons} and \autoref{remark_negative} we can distinguish three cases. If $A_{X'}$ has non-spherical type or has spherical type different from $A$, $D$, $E_6$ and $I_2(m)$, then all the conjugations by the elementary ribbons are trivial. If $A_{X'}$ is of type $E_6$ or $I_2(m)$, then $\Delta_{X'}^2$ is the smallest positive power of $\Delta_{X'}$ that is central and all conjugations are as indicated in the second item of \autoref{lemma:elementary_ribbons}.  If $A_{X'}$ is of type $A$ or $D$, the result is trivial.
\end{proof}

\bigskip

Now we define the two main properties that used ribbons that are conjectured to be true for every Artin group:

\begin{definition}\label{def:ribbon}
Given an Artin group $A_S$ and $S'\subseteq S$, we say that a pair $(X,Y)$, $X,Y\subseteq S'$, is \emph{conjugate by ribbons} in $A_{S'}$ if, for any $g\in A_{S'}$,

$$g^{-1}A_{X}g=A_Y \text{ if and only if }  g\in  A_X\cdot (\mathrm{Ribb}(X,Y)\cap A_{S'}).$$
We say that $A_S$ satisfies the \emph{ribbon property} if, for any two sets of generators $X,Y\subset S$, the pair $(X,Y)$ is conjugate by ribbons in $A_Z$ for every $Z\in \{T\subseteq S \,|\, X,Y \subseteq T\}$.
\end{definition}

\begin{definition}\label{def:standardizable}
Let $A_S$ be an Artin group and $X,Y\subset S$. We say that the pair $(X,Y)$ is \emph{standardisable} in $A_S$ if
$$\forall g\in A_S \text{ such that } g^{-1}A_Yg\subseteq A_X \text{ there are } h\in A_X \text{ and } Z\subseteq X  \text{ such that } h^{-1}g^{-1}A_Ygh=A_Z. $$
In particular, if there is no $g\in A_S \text{ such that } g^{-1}A_Yg\subseteq A_X$, then $(X,Y)$ is standardisable. 
We say that $A_S$ is \emph{standardisable}
if every pair $(X,Y)$, $X,Y\subset S$, is standardisable.
\end{definition}

Godelle conjectures that every Artin group is standardisable and has the ribbon property \citep[Conjecture~1, Conjecture~4.2]{Godelle2} after the first article by \citep{Paris} showing the ribbon property and other results about normalizers for spherical-type Artin groups. Godelle proves that FC-type Artin groups satisfy the ribbon property in \citep[Theorem~3.2]{GodelleFC} and in \citep[Proposition~4.3]{Godelle2} he uses the ribbon property to prove that they are also standardisable. He also shows that all two-dimensional Artin groups are standardisable and satisfy the ribbon property, and this is what we use in \citep*{CMV} to solve the conjugacy stability problem for large Artin groups.

\subsection{Proof of Theorem A}

To prove Theorem~A we will first prove the following theorem: 

\begin{theorem}\label{th:main}
Let $A_S$ be an Artin group and let $X\subset S$. There is an algorithm that decides whether $A_X$ is conjugacy stable if the three following properties hold: 

\begin{itemize}

\item For any $Y\subset S$, the pair $(X,Y)$ is standardisable;

\item For any $X_1,X_2\subseteq X$, the pair $(X_1,X_2)$  is conjugate by ribbons in $A_S$ and in $A_X$;


\item Every element $\alpha \in A_X$ has a parabolic closure $P_\alpha$ in $A_S$. 
\end{itemize}
 

\end{theorem}

\noindent
Let us see that the previous theorem implies Theorem~A:

\begin{proof}[Proof of Theorem~A]
Let $A_S$ be an Artin group. To give a solution to the conjugacy stability problem for parabolic subgroups of Artin groups, we shall notice that the property of being conjugacy stable is preserved under conjugation. Hence, it suffices to give an algorithm that tells if a standard parabolic subgroup $A_X$ is conjugacy stable for every $X\subset S$. In order to satisfy the conditions of Theorem~A, $A_S$ need to be standardisable and conjugate by ribbons and every element $\alpha \in A_S$ has a parabolic closure $P_\alpha$ in $A_S$. In particular, we have the three conditions of \autoref{th:main} for every $X\subset S$, so we have the desired algorithm.

\end{proof}

\noindent
Before showing \autoref{th:main}, we will prove some lemmas:

\begin{lemma}\label{lemma:minimalparabolics}
Let $A_S$ be any Artin group, $X\subset S$ and $\alpha,\beta \in A_S$.  The parabolic closure $P_{\beta^{-1}\alpha \beta}$ exists if and only if $P_\alpha $ exists and in this case $P_{\beta^{-1}\alpha \beta}= \beta^{-1}P_\alpha \beta$.
\end{lemma}

\begin{proof}
Suppose that $P_\alpha$ exists and let $Q$ be a parabolic subgroup containing  $\beta^{-1}\alpha \beta$. Then $\alpha\in \beta Q \beta^{-1}$, where $\beta Q \beta^{-1}$ is a parabolic subgroup. Hence $P_\alpha \subset \beta Q \beta^{-1}$ and $\beta^{-1}P_\alpha \beta\subset Q$. The converse is symmetric. 
\end{proof}

\begin{remark}\label{remark} Let $(*)$ be a property for parabolic subgroups that is preserved under conjugation, such as being of spherical type.
Notice that the previous proof can be adapted to prove that, if $P(*)_\alpha$ is the unique minimal parabolic subgroup containing $\alpha$ and satisfying $(*)$, then $\beta^{-1} P(*)_\alpha \beta$ is the unique minimal parabolic subgroup containing $\beta^{-1}\alpha\beta$ and satisfying $(*)$. 
\end{remark}

The \emph{support}, $\mathrm{supp}(g)$, of a positive element $g\in A_S$ is the set of all generators that appear in any positive word representing $g$. For Artin groups of spherical type, the parabolic closure of an element depends on the element support. In particular, for positive elements we have the following result:

\begin{lemma}[{\citealp[Proposition~6.8]{CGGW}}]\label{lemma:support} Let $A_S$ be an Artin group of spherical type. The parabolic closure of a positive element $g\in A_S$ is $A_{\mathrm{supp}(g)}$.
\end{lemma}

Our strategy to know if two elements are conjugate inside a parabolic subgroup will be based on taking their parabolic closure and verifying if this parabolic closure are conjugate inside the parabolic subgroup. However, there several ways of sending a set of generators to another set of generators by conjugacy. Due to that, standard parabolic subgroups of type $D_k$ will produce special cases that will need to be treated separately. The following three lemmas will help to deal with this cases.

\begin{lemma}\label{lemma:Dn}
Let $A_S$ be a spherical-type Artin group. Let $\Delta^e$ be a central power of the Garside element and $\alpha$ and $\beta$ be two elements of $A_S$. Then $\alpha$ and $\beta$ are conjugate if and only if $\alpha^{-1}\Delta^e$ and $\beta^{-1}\Delta^e$ are conjugate. 

In particular, using the same numbering as in \autoref{coxeter}, the elements $(s_1s_3s_4\cdots s_n)^{-1}\Delta $ and $(s_2s_3s_4\cdots s_n)^{-1}\Delta $ are not conjugate in $D_{n}$ when $n$ is even.

\end{lemma}

\begin{proof}
The first statement is quite straightforward as $\gamma^{-1}\alpha \gamma = \beta$ if and only if $\gamma^{-1}\alpha^{-1} \gamma = \beta^{-1}$. Since $\Delta^e$ is central, this happens if and only if $\gamma^{-1}\alpha^{-1} \Delta^e \gamma = \gamma^{-1}\alpha^{-1} \gamma \Delta^e = \beta^{-1} \Delta^e$.

For the second statement, notice that by \autoref{lemma:support} the parabolic closures of $s_1s_3s_4\cdots s_n $ and $s_2s_3s_4\cdots s_n $ are respectively $A_{S\setminus\{s_2\}}$ and $A_{S\setminus\{s_1\}}$. We know by \myref{Algoritmo:Paris}{Algorithm} that these two parabolic subgroups are not conjugate in $D_n$ if $n$ is even. Therefore, by \autoref{lemma:minimalparabolics} and the previous statement, $(s_1s_3s_4\cdots s_n)^{-1}\Delta $ and $(s_2s_3s_4\cdots s_n)^{-1}\Delta $ are not conjugate.
\end{proof}

\begin{lemma}\label{lemma:D4}

Let $A_S$ be an Artin group and let~$A_X$ and $A_T$ be standard parabolic subgroups so that $T\subset X\subset S$ such that $(T,T)$ is conjugate by ribbons in~$A_X$. Suppose that there is $s\in S\setminus X$ so that the connected component of~$\Gamma_{T\cup \{s\}}$ containing~$s$ is of type~$D_{2k+1}$, $k>1$, and $s=s_5$ with the numbering established for~$D_{2k+1}$ in \autoref{coxeter}. If~$A_X$ is conjugacy stable, then at least one of the following situations applies:

\begin{itemize}

\item There is $t\in X$ adjacent to $s_4$ such that the connected component of $A_{T\cup \{t\}}$ containing $t$ and $X'$ is of type $D_{2k'+1}$, $k'>1$.

\item There are $t_1\in X$ adjacent to $s_1$ and $t_2\in X$  adjacent to $s_2$, such that the connected component of $\Gamma_{T\cup \{t_1\}}$ containing $t_1$ is of type $D_{2k_1+1}$ and the connected component of $\Gamma_{T\cup \{t_2\}}$ containing $t_2$ is of type $D_{2k_2+1}$, $k_1,k_2 >1$.

\end{itemize}

\end{lemma}

\begin{proof} Let $\Gamma_{Y}$ be the connected component of~$A_{T\cup \{s_{5}\}}$ containing~$s$, which is of type $D_{2k+1}$, for some $k>1$. Notice that $s=s_5$ can be adjacent to at most two connected components of~$\Gamma_T$, one being generated the subset $X'=\{s_1,s_2,s_3,s_4\}$ with the numbering established for~$D_{2k+1}$ in \autoref{coxeter}. The element
$\Delta_{Y}$
conjugates $a:=(s_1s_3s_4)^{-1}\Delta_{X'}\Delta_{T\setminus X'}= s_2s_1s_3s_4s_2s_1s_3s_4s_2 \Delta_{T\setminus X'}$ to $b:=(s_2s_3s_4)^{-1}\Delta_{X'} \Delta_{T\setminus X'}= s_1s_2s_3s_4s_1s_2s_3s_4s_1 \Delta_{T\setminus X'}$.  

Since~$A_X$ is conjugacy stable, there is an element $x\in A_X$ such that $x^{-1}ax=b$.
As~$a$ and~$b$ are positive and $\mathrm{supp}(a)=\mathrm{supp}(b)=T$, by \autoref{lemma:support} the parabolic closures of~$a$ and~$b$ are both~$A_{T}$. Then, by \autoref{lemma:minimalparabolics}, any element that conjugates~$a$ to~$b$ normalises~$A_{T}$.  In particular $x^{-1}A_{T}x=A_{T}$ and, since $(T,T)$ is conjugate by ribbons in~$A_{X}$, we can write $x$ as $x=x_1x_2$ where $x_1\in A_{T}$ and $x_2\in A_X$ is a $(T,T)$--ribbon. Equivalently, we can write $x=x_2x_3$, where $x_3=x_2^{-1}x_1x_2\in A_{T}$. 

 The non-trivial elementary $(T,-)$--ribbons in~$A_X$ are the ones written as $r_t:=\Delta_{Z\setminus\{t\}}^{-1}\Delta_{Z}$ or $\Delta_{Z\setminus\{t\}}\Delta_{Z}^{-1}$ (\autoref{remark:negativeribbons}), where $t\in X$, $\Gamma_{Z}$ is the connected component of~$\Gamma_{{T}\cup \{t\}}$ that contains~$t$ and $\Gamma_{Z\cup \{t\}}$ has spherical type. If $\Gamma_Z$ does not contain $X'$, since $\Delta_{T\setminus X'}$ is central in $A_{T\setminus X'}$, the conjugation by all other elementary $(T,-)$--ribbons  will commute with $X'$ and will fix~$a$.
If~$\Gamma_Z$ contains~$X'$, as~$X'$ has type~$D_4$, we have that $\Gamma_{Z\cup \{t\}}$ has  type~$D_m$. In the case of type $D_m$, $m$ even, the conjugation by $r_t$ centralises $A_{Z}$ and, in particular, it fixes $a$. In the case of type~$D_m$ with $m$ odd, $t$ is adjacent to either~$s_1$, or~$s_2$, or $s_4$. If for that case we suppose that none of the items of the lemma are satisfied, then all $t\in X$ are all adjacent to $s_1$ or all adjacent to $s_2$. So suppose without loss of generality that all $t\in X$ are adjacent to~$s_1$. Then, the conjugation by~$r_t$ normalises $A_{Z}$ and permutes the elements in~$Z$ (it switches~$s_2$ and~$s_4$). Hence, $r_t$ conjugates $a$ to $c:=(s_1s_3s_2)^{-1}\Delta_{X'}\Delta_{T\setminus X'}$ and it conjugates~$c$ to~$a$. It follows that, as~$x_2$ is a product of elementary ribbons, each one preserving~$T$, one has $x_2^{-1}ax_2\in\{a,c\}$. Then, since $x_3^{-1}(x_2^{-1}ax_2)x_3=b$ either~$a$ or $c$ are conjugate to~$b$ in~$A_{T}$ (and so in $A_{X'}$). However, we know by \autoref{lemma:Dn} that~$a$ is not conjugate to~$b$ in~$A_{X'}$, and that $c$ is not conjugate to~$b$ in~$A_{X'}$.  A contradiction. Hence some of the items of the statement must be satisfied.
\end{proof}

\begin{lemma}\label{lemma:DK+1}

Let $A_S$ be an Artin group and $A_X$ and $A_T$, $T\subset X\subset S$, be standard parabolic subgroups such that $(T,T)$ is conjugate by ribbons in $A_X$. Suppose that there is $s\in S\setminus X$ so that the connected component of  $\Gamma_{T\cup \{s\}}$ containing $s$ is of type $D_{2k+1}$, $k>2$, and $s=s_{2i+1}$, $i>2$, with the numbering established for~$D_{2k+1}$ in \autoref{coxeter}. If $A_X$ is conjugacy stable in $A_S$, then there is $t\in X$ adjacent to $s_{2i}$ such that the connected component of $\Gamma_{T\cup \{t\}}$ containing $t$ is of type $D_{2k'+1}$, $k'>2$.

\end{lemma}

\begin{proof} The proof is analogous to the proof of the previous lemma. Notice that $s=s_{2i+1}$ can be adjacent to at most two connected components of~$\Gamma_T$, one being generated the subset $X'=\{s_1,s_2,s_3,s_4,\cdots s_{2i}\}$ with the numbering established for~$D_{2k+1}$ in \autoref{coxeter}.
 Suppose also that there is no $t\in X$ adjacent to $s_{2i}$ such that the connected component of $\Gamma_{T\cup \{t\}}$ containing~$X'$ is of type $D_{2k'+1}$ for some $k'>2$.
The element $\Delta_{Y}$ conjugates $a:=(s_1s_3\cdots s_{2i})^{-1}\Delta_{X'} \Delta_{T\setminus X'} $ to $b:=(s_2s_3\cdots s_{2i})^{-1}\Delta _{X'} \Delta_{T\setminus X'}$. Suppose that $A_X$ is conjugacy stable. Then, there is $x\in A_X$ such that $x^{-1}ax=b$.

It is well-known by experts that $\Delta_{X'}= (s_2s_3\cdots s_{2i}s_1)^{2i-1}=(s_1s_3\cdots s_{2i}s_2)^{2i-1}$ (see \citealp{BS} and \citealp{Paris3}). So, by \autoref{lemma:support}, the parabolic closure of both~$a$ and~$b$ is~$A_{T}$. Then, by \autoref{lemma:minimalparabolics}, $x$ normalises~$A_{T}$. As in the proof of the previous lemma, $(T,T)$ being conjugate by ribbons in~$A_X$ means that we can write $x=x_2x_3$, where $x_2\in A_X$ is an $(T,T)$--ribbon and $x_3\in A_{T}$. The non-trivial elementary {$(T,-)$--ribbons} belonging to~$A_{X}$ are $r_t:=\Delta_{Z\setminus\{t\}}^{-1}\Delta_{Z}$ or $\Delta_{Z\setminus\{t\}}\Delta_{Z}^{-1}$ (\autoref{remark:negativeribbons}), where $t\in X$, $\Gamma_{Z}$ is the connected component of $\Gamma_{{T}\cup \{t\}}$ that contains~$t$ and $\Gamma_{Z\cup \{t\}}$ has spherical type. If~$\Gamma_Z$ does not contain $X'$, since $\Delta_{T\setminus X'}$ is central in~$A_{T\setminus X'}$, the conjugation by all other elementary $(T,-)$--ribbons will commute with~$X'$ and will fix~$a$. Otherwise, in~$\Gamma_{Z}$ has type~$E_7$ or~$D_{2k'+1}$ for some $k'>i$. In the~$E_7$ case, the conjugation by~$r_t$ centralises~$A_{Z}$ and  it fixes~$a$. We have supposed that the $D_{2k'+1}$ case does not happens, so $x_2^{-1}ax_2=a$. Hence $x_3^{-1}ax_3=b$ with $x_3\in A_{X'}$, which by \autoref{lemma:Dn} is a contradiction. 
 \end{proof}

The following lemma will help to deal with fact that an element in the stabilizer of a standard parabolic subgroup can permute its connected components.

\begin{lemma}\label{lemma:permutations}
Let $A_S$ be an Artin group and $A_X$, $X\subset S$, a standard parabolic subgroup of $A_S$ which is conjugacy stable. Suppose that for any $X_1,X_2\subseteq X$, the pair $(X_1,X_2)$  is conjugate by ribbons in $A_X$. Let $X_1,X_2\subset X$ be such that $g^{-1}{X_1}g={X_2}$, for some $g\in A_S$. Then there is $g'\in A_X$ such that $g^{-1}{Y}g=g'^{-1}{Y}g'$ for every connected component $\Gamma_{Y}$ of $\Gamma_{X_1}$.
\end{lemma}

\begin{proof}

Denote by $\Gamma_{Y_1},\Gamma_{Y_2},\dots ,\Gamma_{Y_l}$, $Y_i\subset X_1$, the connected components of $\Gamma_{X_1}$. 
For each $i$, denote by $s_{i,j}$, $1\leq j \leq |Y_i|$ the elements in $Y_i$. Now define $y_i:=\left(\prod_{j=1}^{|Y_i|} s_{i,j}\right)^{k_i}$, where the $k_i$'s are chosen to satisfy that the number of letters in $\left(\prod_{j=1}^{|Y_i|} s_{i,j}\right)^{k_i}$ is different from the number of letters in $\left(\prod_{j=1}^{|Y_{i'}|} s_{i',j}\right)^{k_{i'}}$ for $i'\neq i$. 

%
%
%

\medskip


By \autoref{lemma:support}, we know that the element $a:=y_1y_2\dots y_l$ has parabolic closure $A_{X_1}$ and that $g^{-1}ag$ has parabolic closure $A_{X_2}$. Since $A_X$ is conjugacy stable in $A_S$, there must be $h\in A_X$ such that $h^{-1}ah=g^{-1}ag$. Then, by \autoref{lemma:minimalparabolics}, $h^{-1}A_{X_1}h=A_{X_2}$. The ribbon hypothesis tells us that $h=h_1h_2$ with $h_1\in A_{X_1}$ and $h_2\in \mathrm{Ribb}(X_1,X_2)\cap A_X$. So $h_2$ is an element in $A_X$ conjugating $X_1$ to $X_2$. Also, since $h_1\in A_{X_1}$, the conjugation $h_1$ preserves each $Y_i$ and one has that $h^{-1}A_{Y_i}h=h_2^{-1}A_{Y_i}h_2$ for $1\leq i \leq l$ --observe that conjugations by $g$ and $h_2$ send letters to letters, but $h$ and $h_1$ do not have to--.

\medskip

Suppose that  $g^{-1}{Y_i}g\neq h_2^{-1}{Y_i}h_2$ for some $1\leq i\leq l$. By the previous discussion, one should have that $hg^{-1}(y_1 y_2\dots y_l)gh^{-1}= y_1 y_2\dots y_l$, but let us see that this is impossible: The conjugation by $gh^{-1}$ permutes non trivially the components of $A_{X_1}$ and every $y_i$ is contained in a different component, so
 there is some $y_i$ such that $hg^{-1}y_i gh^{-1}= y_{i'}$, for $i'\neq i$. However, since the relations in an Artin group are homogeneous, two positive words representing two conjugate positive elements need to have the same number of letters. This is a contradiction and therefore we can set ${g'=h_2}$.
\end{proof}

\medskip

\begin{proof}[Proof of \autoref{th:main}]

\medskip
We are going to prove that a standard parabolic subgroup $A_X$ is conjugacy stable in $A_S$ if and only if the following conditions are fulfilled: 

\begin{enumerate}
\item For every $X'\subset X$ such that $g^{-1}{X'}g\subset {X}$, for some $g\in A_S$, there is $h\in A_X$ such that $g^{-1}{X''}g=h^{-1}{X''}h$ for every connected component $\Gamma_{X''}$ of $\Gamma_{X'}$.

\item Let $A_T$, $ T\subset X$, be a parabolic subgroup. Let  $s\in S\setminus X$ be such that the connected component of~$\Gamma_{T\cup \{s\}}$ containing $s$ is of type~$D_{2k+1}$, $k> 2$ with $s=s_{2i+1}$ following the numbering of \autoref{coxeter}, then there is $s'\in X$  adjacent to $s_{2i}$ such that the connected component of $\Gamma_{T\cup \{s'\}}$ containing $s'$ is of type $D_{2k'+1}$, $k'>2$. This condition is checked by \myref{Algoritmo:Dn}{Algorithm}.

\item Let $A_T$, $T\subset X$, be a parabolic subgroup. Let $s\in S\setminus X$ be such that the connected component of $\Gamma_{T\cup \{s\}}$ containing~$s$ is of type~$D_{2k+1}$, $k>1$ and $s=s_5$ with the numbering of \autoref{coxeter}. Then either there is $s'\in X$ adjacent to $s_4$ such that the connected component of $\Gamma_{T\cup \{s'\}}$ containing $s'$ is of type~$D_{2k'+1}$, $k>1$, or there are $t_1\in X$ adjacent to~$s_1$ and~$t_2\in X$ adjacent to~$s_2$  such that the connected component of $\Gamma_{T\cup \{t_1\}}$ containing $t_1$ is of type $D_{2k_1+1}$, $k_1>1$ and the connected component of $\Gamma_{T\cup \{t_2\}}$ containing $t_2$ is of type $D_{2k_2+1}$, $k_2>1$. This condition is checked by \myref{Algoritmo:D4}{Algorithm}.

\end{enumerate}
 
Once this is proven, we will be able to construct an algorithm to solve the conjugacy stability problem, explained in \myref{Algoritmo:main}{Algorithm}. This is a refinement of \myref{Algoritmo:Paris}{Algorithm} that considers permutations of components and the $D_{2k}$ exceptions.

\medskip

By \autoref{lemma:D4}, \autoref{lemma:DK+1} and \autoref{lemma:permutations}, we know that if some of the items is not satisfied, then $A_X$ cannot be conjugacy stable in $A_S$. So we need to prove that if all the items are satisfied, then $A_X$ is conjugacy stable. 

\bigskip

Let $\alpha,\beta\in A_X$ be such that there is $g\in A_S$ satisfying $g^{-1}\alpha g = \beta$.
Let $P_\alpha,P_\beta\subset A_X$ be the minimal parabolic subgroups containing $\alpha$ and $\beta$ respectively. Suppose that all of the items of the theorem are fulfilled.
Thanks to the standardisation condition, we know that $P_\alpha = g_1^{-1}A_{Y}g_1$ and $P_\beta = g_2^{-1}A_{Z}g_2$, with $g_1,g_2\in A_X$ and $Y,Z\subseteq X$.   Also, \autoref{lemma:minimalparabolics} implies that $P_\alpha=g^{-1}P_\beta g$. So, up to conjugacy by elements of $A_X$, we can suppose that $A_Y$ and $A_Z$ are the parabolic closures of $\alpha$ and $\beta$ and are conjugate by $g$. 

\medskip

$(Y,Z)$ being conjugate by ribbons in $A_S$ tells us that $g=a_1\cdot a_2$ where $a_1\in A_Y$ and $a_2$ is a $(Y,Z)$--ribbon in $A_S$.
Since the first item is satisfied and $a_1^{-1}Y a_1 = Z$, we know that there is $h\in A_X$ such that $g^{-1}{Y_i}g=h^{-1}{Y_i}h$ for every connected component $\Gamma_{Y_i}$ of $\Gamma_{Y}$. As $(Y,Z)$ is conjugate by ribbons in $A_X$, $h=b_1\cdot b_2$ where $b_1\in A_Y$ and $b_2$ is a $(Y,Z)$--ribbon in $A_S$. Also, note that $b_1$ cannot permute the connected components of $Y$: $b_1=y_1\cdots y_n$ where $y_i\in Y_i$, so every $y_i$ commutes with $Y_j$, $j\neq i$. Hence $a_2^{-1}{Y_i}a_2=b_2^{-1}{Y_i}b_2$ for every connected component $\Gamma_{Y_i}$ of $\Gamma_{Y}$.

\medskip
Notice that $a_2b_2^{-1}$ is a $(Y,Y)$--ribbon. Since the connected components of $Y$ are preserved, the conjugation by $a_2b_2^{-1}$ induces an isomorphism of the subgraph $\Gamma_{Y_i}$.
If~$A_{Y_i}$ has not spherical type or it has spherical type and it is non-twistable, we know by \autoref{lemma:permutation} that  $b_2a_2^{-1}sa_2b_2^{-1}=s$, for every $s\in Y_i$.  If~$A_{Y_i}$ is of type $A$ or $D_{2k+1}$, the only isomorphisms of graphs that we can have are the trivial one and a reflection switching the vertices corresponding to the two first generators. Then we have either $b_2a_2^{-1}sa_2b_2^{-1}=s$ or $\Delta_{Y_i}^{-1}b_2a_2^{-1}sa_2b_2^{-1}\Delta_{Y_i}=s$. This is also valid if $A_{Y_i}$ has type $E_6$ or $I_2(m)$ by \autoref{lemma:permutation}. 
If $A_{Y_i}$ is of type~$D_{2k}$, the only non trivial isomorphism of~$\Gamma_{Y_i}$ is a switch of the two first vertices. But, since $\Delta_{Y_i}$ is central, the conjugation by~$\Delta_{Y_i}$ cannot perform this isomorphism. Then, if the conjugation by $a_2b_2^{-1}$ switches the vertices, there must be a $t\in S$ adjacent to~$\Gamma_{Y_i}$ such that the connected component of $\Gamma_{Y\cup \{t\}}$ containing $Y_i$ is of type $D_{2k'+1}$ (notice that the conjugation by $\Delta_{2k'+1}$ does the desired switching). Then, by conditions 2 and 3 above, either there is $t\in X$ adjacent to $Y_i$ such that $\Delta_{Y'}^{-1}b_2a_2^{-1}sa_2b_2^{-1}\Delta_{Y'}=s$, where $\Gamma_{Y'}$ is the connected component of $\Gamma_{Y\cup\{t\}}$ containing $t$; or there are $t_1,t_2\in X$ adjacent to $Y_i$ such that $\Delta_{Y''}^{-1}\Delta_{Y'}^{-1}b_2a_2^{1}sa_2b_2^{-1}\Delta_{Y'}\Delta_{Y''}=s$, where $\Gamma_{Y'}$ is the connected component of $\Gamma_{Y\cup\{t_1\}}$ containing~$\Gamma_{Y'}$ is the connected component of $\Gamma_{Y\cup\{t_2\}}$ containing $\Gamma_{Y'}$.

\medskip
 The previous paragraph means that we can suppose that $a_2^{-1}s a_2 = b_2^{-1}sb_2$ up to conjugations by elements of the form $\Delta_{X'}$, $X'\subset X$. Then, up these conjugations, $g^{-1}\alpha g = a_1^{-1}b_2^{-1}\alpha b_2 a_1$. Since $b_2,a_1\in A_X$, we have proven that under the three items, $A_X$ is conjugacy stable in $A_S$.
\end{proof}

\section{FC-type Artin groups}\label{Section4}

If every standard parabolic subgroup of an Artin group $A_S$ that does not contain infinite relations has spherical type, then~$A_S$ is said to be of FC-type. The aim of this section is to prove Theorem~B, that is, to discuss whether we can apply the algorithm to solve the conjugacy stability problem in Artin groups of FC-type. 
From now on, suppose that~$A_S$ has FC-type. 

\smallskip

It is shown in \citep[Proposition~2]{Altobelli} that, if $s,t\in S$ are such that $m_{s,t}=\infty$, then~$A_S$ is isomorphic to the amalgamated free product of~$A_{S\setminus\{s\}}$ and $A_{S\setminus\{t\}}$ over $A_{S\setminus\{s,t\}}$, denoted by $A_{S\setminus\{s\}} *_{A_{S\setminus\{s,t\}}} A_{S\setminus\{t\}}$. Then, if we give an order to the $\infty$-labels of $A_S$, we will obtain a specific amalgamated product structure for $A_S$. The next two propositions about canonical forms in amalgamated free products can be found in \citep*[Section~4]{MKS} and sometimes will be used without explicit reference:

\begin{proposition}[Canonical form for amalgamated free products]
 Given the amalgamated free product $G=G_1 *_H G_2$ of the groups $G_1$ and $G_2$ over $H$, we respectively denote by $C_1$ and $C_2$ the transversals of $G_1/H$ and $G_2/H$ that contain $1$.  Then, every $x\in G$ can be uniquely expressed as a product $x=x_1x_2\cdots x_r h$, where $h\in H$, $x_i\in C_1\cup C_2$ is not trivial for $i=1,\dots, r$, and $x_i$ and $x_{i+1}$ belong to different transversals if $r>1$. This expression is called the \emph{amalgam normal form} of $x$, and we denote it by $\rho(x)$.  

\end{proposition}

\begin{proposition}[Conjugacy in amalgamated free products]\label{lemma:cr}
Given the previous amalgamated free product $G=G_1 *_H G_2$, any element $g\in G$ is conjugate to an element $x$ with amalgam normal form $x_1x_2\cdots x_r h$, in which $x_1$ and $x_r$ belong to different transversals. We say that such an element $x$ is \emph{cyclically reduced}. Moreover, \begin{itemize}

\item if $x$ is conjugate to a element written $y=p_1p_2p_3\cdots p_k$, $k\geq 2$, where $p_i$,$p_{i+1}$ as well as $p_1$,$p_{k}$ belong to different transversals, then $x$ is obtained from $y$ by cyclically permuting $p_1p_2p_3\cdots p_k$ and then conjugating by an element of $H$; 

\item if $H=\{1\}$ and $x$ is conjugate to an element $y$ in one of the factors ($G_1$ or $G_2$), then $x$ and $y$ belong to the same factor and are conjugate in that factor.

\end{itemize}

\end{proposition}

\noindent
The amalgamated free product structure of a standard parabolic subgroup $A_X$ of $A_S$ heavily relies on the structure of $A_S$. The next result is a consequence of \citep[Theorem~2]{Altobelli}.

\begin{lemma}[{\citealp[Corollary~1.12]{GodelleFC}}]\label{lema:amalgama}
Let $A_S\simeq A_{S\setminus\{s\}} *_{A_{S\setminus\{s,t\}}} A_{S\setminus\{t\}}$ be a $FC$-type Artin group. Let $X\subset S$. If $w\in A_X$, then the amalgam normal form of $w$ has its terms in $A_X$.
\end{lemma}

\subsection{Proof of Theorem B}

In \citep[Corollary~3.2]{Rose}, it is proven that if an element $\alpha$ of an FC-type Artin group is contained in a spherical-type parabolic subgroup, then there is a unique minimal (by inclusion) spherical-type parabolic subgroup $Q_\alpha$ containing $\alpha$. We call $Q_\alpha$ the \emph{spherical-type parabolic closure} of $\alpha$. Since all parabolic subgroups of a spherical-type Artin group have spherical type, we can use \autoref{remark} to easily adjust the proof of \autoref{th:main} and be able to apply \myref{Algoritmo:main}{Algorithm} to spherical-type parabolic subgroups.  

\medskip
\noindent
The following lemma will allow us to complete the proof of the first part of Theorem~B.


\begin{lemma}\label{lema:parabolicoenaesferico}
Let $A_S$ be an Artin group of FC type and $A_X$, $X\subset S$, a  standard parabolic subgroup of non-spherical type. If $\alpha\in A_X$ belongs to a spherical parabolic subgroup of $A_S$, then the spherical-type parabolic closure $Q_\alpha$ of~$\alpha$ is contained in~$A_X$. 
\end{lemma}

\begin{proof}

Choose an $\infty$-label $m_{s,t}$ in $A_X$ and take the decompositions $A\simeq A_{S\setminus\{s\}} *_{A_{S\setminus\{s,t\}}} A_{S\setminus\{t\}}$ and $A_X\simeq A_{X\setminus\{s\}} *_{A_{X\setminus\{s,t\}}} A_{X\setminus\{t\}}$. We know by \autoref{lema:amalgama} that the amalgam normal form of~$\alpha$ has their terms in $A_X$ so it is also an amalgam normal form with respect to the structure of~$A_X$. Then, we can obtain a cyclically reduced element~$x\in A_X$ from~$\alpha$ by conjugating by an element~$c$ of~$A_X$. Also, we can write $Q_\alpha = \beta^{-1}A_Y\beta$, where~$A_Y$ is a spherical type standard parabolic subgroup of~$A_S$. Then, $\beta \alpha \beta^{-1} \in A_Y$ and we can obtain a cyclically reduced element~$y\in A_Y$ from $\beta \alpha \beta^{-1} $ by conjugating by an element of~$A_Y$.  We will show our lemma by induction on the number of $\infty$-labels of~$A_X$.  

\medskip

Suppose that there is only one $\infty$-label $m_{s,t}$ in~$A_X$. We first prove that $x$ is contained in a spherical-type standard parabolic subgroup $A_{X'}$. In this case $A_{X\setminus\{s\}}$ and $A_{X\setminus\{t\}}$  have spherical type, so if $x$ is contained in any of them we are done. Suppose then that~$x$ is not contained in any of these two subgroups. As~$x$ and~$y$ are conjugate and cyclically reduced, by \autoref{lemma:cr} $x$~is obtained from~$y$ by conjugating by an element in $A_{Y \cup (S\setminus\{s,t\})}$. Then, $x\in A_{X'}:= A_{Y \cup (S\setminus\{s,t\})}$. Since~$A_Y$ has spherical type, $Y$~cannot simultaneously contain $s$ and $t$. By \citep{Vanderlek}, the intersection of standard parabolic subgroups is (the expected) standard parabolic subgroup, meaning that $A_X\cap A_{X'}=A_{X\cap X'}$. So $x$ is contained in $A_{X\cap X'}$, which has spherical type because it lies in $A_X$ and cannot contain simultaneously $s$ and $t$. Conjugating by~$c^{-1}$, we have that~$\alpha$ is in the spherical-type parabolic subgroup $c A_{X\cap X'} c^{-1}<A_X$, which must contain~$Q_\alpha$ because the spherical-type parabolic closure is unique. This finishes the proof of the base case of our induction.

\medskip

To prove the step case suppose that, if $\alpha$ is contained in a standard parabolic subgroup with less than $k$ $\infty$-labels, then $Q_\alpha$ is contained in that parabolic subgroup. Let $A_X$ have $k$ labels. If $x$ belongs to $A_{X\setminus\{s\}}$ or $A_{X\setminus\{t\}}$, then $x$ belongs to the standard parabolic subgroup containing less than $k$ $\infty$-labels. Otherwise, applying the same reasoning as in the initial case,  $x\in A_X \cap A_{Y \cup (S\setminus\{s,t\})}$, which also has less than $k$ $\infty$-labels. Thus, by hypothesis, the spherical-type parabolic closure $Q_x$ of $x$ is in $A_X$. 
Therefore, $\alpha$ is in the spherical-type parabolic subgroup $c Q_x c^{-1}<A_X$, which must contain $Q_\alpha$.
\end{proof}

In the particular case in which $A_S$ is a free product of spherical-type Artin groups, we can prove the existence of a parabolic closure, hence all the hypothesis of Theorem~A will be fulfilled.

\begin{lemma}\label{lemma:power}
Suppose that $A_S$ is an Artin group of FC-type such that  $A_S\simeq A_{X_1}*A_{X_1}*\dots * A_{X_k} $, where every $A_{X_i}$ is a spherical-type Artin group. Let $\alpha\in A_S$. Then, any minimal parabolic subgroup containing $\alpha^m$ for any $m\in \mathbb{Z}$ contains also $\alpha$.
\end{lemma}

\begin{proof}
If $\alpha$ in contained in a single factor $A_{X_i}$, this is proven in \citep[Theorem~8.2]{CGGW}. Suppose otherwise and let $P$ be a minimal parabolic subgroup containing $\alpha^m$. There is an element $\beta$ such that $\beta^{-1}P\beta = A_X$ is standard. Notice that $\beta^{-1}\alpha^m \beta = (\beta^{-1}\alpha \beta)^m$. This means that the amalgam normal form of $(\beta^{-1}\alpha \beta)^m$ can be written using only letters in $X$ (\autoref{lema:amalgama}). By hypothesis, the length of the amalgam form of $\beta^{-1}\alpha \beta$ is bigger than 1, hence all the letters in the amalgam normal form of $\beta^{-1}\alpha \beta$ are letters that appear in the amalgam normal form of $(\beta^{-1}\alpha \beta)^m$. Therefore $A_X$ contains $\beta^{-1}\alpha \beta$. Conjugating by $\beta^{-1}$, we have that~$P$ contains $\alpha$.
\end{proof}

\begin{proposition}\label{Prop:clausura_free}
If $A_S$ is an Artin group of FC-type such that  $A_S\simeq A_{X_1}*A_{X_1}*\dots * A_{X_k} $, where every $A_{X_i}$ is a spherical-type Artin group, then every element $\alpha$ has a parabolic closure~$P_\alpha$.

\end{proposition}

\begin{proof}
We will prove the proposition by induction on $k$. If $k=1$, $A_S$ has spherical type and the result is true by \citep[Proposition~7.2]{CGGW}. Now suppose that the result is true for~$k-1$ and consider the free product structure $A_{X_1} * B$ where $B= A_{X_2}*A_{X_3}*\dots * A_{X_k}$. Suppose there are two minimal parabolic subgroups $P_1=\beta^{-1}A_Y \beta$, $P_2=\gamma^{-1}A_Z\gamma$ containing~$\alpha$. By \citep[Theorem~3.1]{Rose}, if $P_1$ and $P_2$ have spherical type, then $\alpha$ is contained in $P_1\cap P_2$, so by minimality $P_1=P_2$. Suppose then that $P_1$ has  non-spherical type. Then, $A_Y$~is a minimal parabolic subgroup containing  $\alpha':=\beta \alpha \beta^{-1}$ and $A_Z$ is a minimal parabolic subgroup containing $\alpha'':=\gamma \alpha \gamma^{-1}$. 
Applying \myref{Algoritmo:Paris}{Algorithm} implies that if~$A_Y$ and~$A_Z$ are different, they cannot be conjugate. 

\medskip

Let $\alpha_1\alpha_2\alpha_3\cdots \alpha_r$ be the amalgam normal form of~$\alpha'$ with respect to $A_{X_1} * B$. We also know that~$\alpha'$ and~$\alpha''$ are conjugate and that all $\alpha_i$'s are contained in~$A_Y$ (\autoref{lema:amalgama}). If $r=1$, then by \autoref{lemma:cr} we have that~$\alpha'$ and~$\alpha''$ belong to the same factor~$F$ of the free product and are conjugate by an element~$f$ in that factor. By inductive hypothesis, $A_Y$ is the parabolic closure of~$\alpha'$ in~$F$ and~$A_Z$ is the parabolic closure of~$\alpha''$ in $F$, so by \autoref{lemma:minimalparabolics} $f$ has to conjugate~$A_Y$ to~$A_Z$, which is only possible if $A_Y=A_Z$. Since $\alpha''= \gamma \beta^{-1}\alpha'\beta \gamma^{-1}$, we can apply again \autoref{lemma:minimalparabolics} to obtain $\gamma\beta^{-1} A_Y \beta\gamma^{-1} = A_Y$, so $P_1=P_2$. If $r\geq 2$, then  $\alpha''$ is obtained from~$\alpha'$ by cyclically permuting the $\alpha_i$'s. This means that $\alpha',\alpha''$ belong $A_Y \cap A_Z$, which by \citep{Vanderlek} is a parabolic subgroup contained in~$A_Y$ and~$A_Z$. As~$A_Y$ and~$A_Z$ are minimal, we have that $A_Y=A_Y \cap A_Z=A_Z$. It remains to show that this implies $P_1=P_2$. Notice that $P_2$ can be obtained from $P_1$ by using conjugation by an element that centralizes $\alpha$, namely $c:=\beta g \gamma^{-1}$, where $g$ is the element that conjugates $\alpha'$ to $\alpha''$. Now, by \citep[Corollary~4.1.6]{MKS}, either $c$ and $\alpha$ are in the same factor (this would be the case $r=1$) or~$\alpha$ and $c$ are a power of the same element $h$. By \autoref{lemma:power}, $h\in P_1$, hence $P_2=c^{-1}P_1c=P_1$.
\end{proof}

\begin{proof}[Proof of Theorem~B]
Thanks to \citep[Theorem~3.2]{GodelleFC} and \citep[Proposition~4.3]{Godelle2}, we know that $A$ is standardisable and has the ribbon property. If $A$ has a free product structure, then every element has a parabolic closure (\autoref{Prop:clausura_free}) and we can apply \myref{Algoritmo:main}{Algorithm}. Now suppose that $A$ is any FC-type Artin group and that $A_X$ is standard parabolic subgroup of $A$. We need to prove that there is an algorithm that takes as input~$A$ and~$A_X$ and  decides whether for every two elements $x,y\in A_X$, with $x,y$ contained in (possibly different) spherical-type parabolic subgroups, and such that $g^{-1}xg=y$ with $g\in A$, there is $g'\in H$ such that $g'^{-1}xg'=y$. \autoref{lema:parabolicoenaesferico} proves that the spherical-type parabolic closures~$Q_x$ and~$Q_y$, are contained in~$A_X$. This last condition and the existence of a spherical-type parabolic closure suffice to reproduce the proof of \autoref{th:main} --just replacing parabolic closures by spherical-type parabolic closures-- and show that running \myref{Algoritmo:main}{Algorithm} will do the job --notice that the only distinct irreducible standard parabolic subgroups that can be conjugate are the spherical-type ones--.

\end{proof}

\bigskip

\noindent{\textbf{\Large{Acknowledgments}}}

The idea of writing this paper came to me while doing a collaboration with Alexandre Martin, to whom I am very grateful for the year I spent in Edinburgh working under his supervision. 
Thanks to Yago Antolín for useful discussions about basics on amalgamated free products. Thanks to Juan Gonz\'alez-Meneses for reading this paper, his suggestions and numerous helpful conversations. I also very much appreciate the comments and remarks made by the referee of this article.

\medskip

\begin{algorithm}[ht]
\SetKwInOut{Input}{Input}\SetKwInOut{Output}{Output}

\BlankLine

\Input{The Coxeter graph $\Gamma_S$ of an Artin group $A_S$ and three subgraphs $\Gamma_X\subset \Gamma_S$, $\Gamma_{Y'}\subset \Gamma_Y\subset \Gamma_X$ such that $A_X$ and $A_S$ satisfies the hypotheses of \autoref{th:main} and $\Gamma_Y'$ is a connected component of $\Gamma_{Y}$ of type $D_{2k}$.}
\Output{$1$ (if we know that $A_X$ is not conjugacy stable) or $0$.}
\BlankLine
\BlankLine

Label the elements $s_1,s_2,\dots,s_{2k}$ of $Y$ as in \autoref{coxeter}.

\For{$t\in \mathrm{Adj}(\{s_{2k}\})\cap (S\setminus X)$}
{

\If{the connected component of $\Gamma_{Y\cup\{t\}}$ containing $Y'$ (and $t$) is of type $D_{2m+1}$, for some $m$}
{ 

\For{$t'\in \mathrm{Adj}(\{s_{2k}\})\cap X$}{
\If{the connected component of $\Gamma_{Y\cup\{t'\}}$ containing $Y'$ (and $t'$) is of type $D_{2m'+1}$, for some $m'$}{ 

\Return{$0$};}

}
\Return{$1$};

}

}

\Return{$0$}

\caption{Algorithm to check the $D_{2k}$, $k>2$, exceptions described in the proof of \autoref{th:main}}
\label{Algoritmo:Dn}
\end{algorithm}

\begin{algorithm}[H]
\SetKwInOut{Input}{Input}\SetKwInOut{Output}{Output}

\BlankLine

\Input{The Coxeter graph $\Gamma_S$ of an Artin group $A_S$ and two subgraphs $\Gamma_X\subset \Gamma_S$, $\Gamma_{Y'}\subset\Gamma_Y\subset \Gamma_X$ such that $A_X$ and $A_S$ satisfy the hypotheses of \autoref{th:main} and $\Gamma_{Y'}$ is a connected component of $\Gamma_Y$ of type $D_{4}$.}
\Output{$1$ (if we know that $A_X$ is not conjugacy stable) or $0$.}
\BlankLine
\BlankLine

Label the elements $s_1,s_2,s_3,s_{4}$ of $Y$ as in \autoref{coxeter}.

$Z=\{s_1,s_2,s_3\}$;

\For{$s\in Z$}
{

\For{$t\in \mathrm{Adj}(\{s\})\cap (S\setminus X)$}
{
$p=0$; $q=0$;

\If{the connected component of $\Gamma_{Y\cup\{t\}}$ containing $Y'$ (and $t$) is of type $D_{2m}$, for some $m$}
{ $p=1$; $q=1$;

\For{$t'\in \mathrm{Adj}(\{s\})\cap X$}{

\If{the connected component of $\Gamma_{Y\cup\{t'\}}$ containing $Y'$ (and $t'$) is of type $D_{2m'+1}$, for some $m'$}{ 

$p=0;$ \, break loop;}

}

\If{$p=1$}{
\For{$t_1\in \mathrm{Adj}(Z\setminus \{s\})\cap  X$}{

\If{the connected component of $\Gamma_{Y\cup\{t_1\}}$ containing $Y'$ (and $t_1$) is of type $D_{2m_1+1}$, for some $m_1$}{

\For{$t_2\in \mathrm{Adj}(Z\setminus \{s,t_1\})\cap  X$}{

\If{the connected component of $\Gamma_{Y\cup\{t_2\}}$ containing $Y'$ (and $t_2$) is of type $D_{2m_2+1}$, for some $m_1$}{

$p=0$; \, break loop;

}}

}

\If{$p=0$}{ break loop;}

}}}

\If{$p=1$}
{\Return{$1$};}

\If{$q=1$}
{break loop;}

}
}

\Return{$0$}

\caption{\small
 Algorithm to check the $D_{4}$ exceptions described in the proof of \autoref{th:main}}
\label{Algoritmo:D4}
\end{algorithm}

\newpage

\begin{algorithm}[H]
\SetKwInOut{Input}{Input}\SetKwInOut{Output}{Output}

\BlankLine

\Input{The Coxeter graph $\Gamma_S$ of an Artin group $A_S$ and a $\Gamma_X\subset A_S$ such that $A_X$ and $A_S$ satisfy the hypotheses of \autoref{th:main}.}
\Output{``$A_X$ is conjugacy stable" or ``$A_X$ is not conjugacy stable''.}
\BlankLine
\BlankLine

\For{$(X_1,X_2)\subset (X,X)$ such that $|X_1|=|X_2|$}{

\If{$\Gamma_{X_1}$ is of type $D_{2k}$}
{

\If{$k>2$}{run \autoref{Algoritmo:Dn};
 
\If{ \autoref{Algoritmo:Dn} returns $1$}
{\Return``$A_X$ is not conjugacy stable";}
}

\If{$k=2$}{run \autoref{Algoritmo:D4};
 
\If{ \autoref{Algoritmo:D4} returns $1$}
{\Return``$A_X$ is not conjugacy stable";}
}

}

$\Gamma_{X'_1}, \Gamma_{X'_2}, \dots,\Gamma_{X'_m}:=$  components of $\Gamma_{X_1}$;

$C:=\{(X'_1,X'_2,\dots,X'_m)\}$;

\If{$X_1=X_2$}{ $D:=\{(X'_1,X'_2,\dots,X'_m)\}$; }

\Else{
$D:=\{\emptyset\}$;}

\For{$(Y_1,Y_2,\dots,Y_m)\in C$}{

$Y:= Y_1 \cup Y_2 \cup \cdots \cup Y_m$;

\For{$t\in X\cap\mathrm{Adj}(Y)$}{
\smallskip

\If{the connected component $\Gamma_{Y'}$ of $\Gamma_{Y\cup \{t\}}$ containing $t$ is twistable}{

\smallskip

$Z=\Delta_{Y'}^{-1}Y\Delta_{Y'} $;

$T=(\Delta_{Y'}^{-1}Y_1\Delta_{Y'},\Delta_{Y'}^{-1}Y_2\Delta_{Y'},\dots,\Delta_{Y'}^{-1}Y_m\Delta_{Y'}) $;

\smallskip

\If{$T\notin C$}{
\smallskip

$C =C \cup \{T\}$;

\If{$Z=X_2$ and $T\not\in D$}{

$D=D\cup \{T\}$;

}}}}

}

\For{$(Y_1,Y_2,\dots,Y_m)\in C$}{

$Y:= Y_1 \cup Y_2 \cup \cdots \cup Y_m$;

\For{$t\in \mathrm{Adj}(Y)$}{
\smallskip

\If{the connected component $\Gamma_{Y'}$ of $\Gamma_{Y\cup \{t\}}$ containing $t$ is twistable}{

\smallskip

$Z=\Delta_{Y'}^{-1}Y\Delta_{Y'} $

$T=(\Delta_{Y'}^{-1}Y_1\Delta_{Y'},\Delta_{Y'}^{-1}Y_2\Delta_{Y'},\dots,\Delta_{Y'}^{-1}Y_m\Delta_{Y'}) $

\smallskip

\If{$T\notin C$}{
\smallskip

$C =C \cup \{T\}$;

\If{$Z=X_2$ and $T\not\in D$}{

\Return{``$A_X$ is not conjugacy stable'';}

}}}}}

}

 \Return{``$A_X$ is conjugacy stable'';}
\caption{Algorithm that tell us if a parabolic subgroup is conjugacy stable or not.}
\label{Algoritmo:main}
\end{algorithm}

\medskip
\bibliography{Bib_CS}

\bigskip\bigskip{\footnotesize%

\noindent
\textit{\textbf{María Cumplido} \\ 
Instituto de Matemáticas de la Universidad de Sevilla (IMUS)\\
Departamento de Álgebra, Universidad de Sevilla, Spain.} \par
 \textit{E-mail address:} \texttt{\href{mailto:cumplido@us.es}{cumplido@us.es}}
 }

\end{document}

%% file: coxeter.pdf_tex
\begingroup%
  \makeatletter%
  \providecommand\color[2][]{%
    \errmessage{(Inkscape) Color is used for the text in Inkscape, but the package 'color.sty' is not loaded}%
    \renewcommand\color[2][]{}%
  }%
  \providecommand\transparent[1]{%
    \errmessage{(Inkscape) Transparency is used (non-zero) for the text in Inkscape, but the package 'transparent.sty' is not loaded}%
    \renewcommand\transparent[1]{}%
  }%
  \providecommand\rotatebox[2]{#2}%
  \ifx\svgwidth\undefined%
    \setlength{\unitlength}{578.70186462bp}%
    \ifx\svgscale\undefined%
      \relax%
    \else%
      \setlength{\unitlength}{\unitlength * \real{\svgscale}}%
    \fi%
  \else%
    \setlength{\unitlength}{\svgwidth}%
  \fi%
  \global\let\svgwidth\undefined%
  \global\let\svgscale\undefined%
  \makeatother%
  \begin{picture}(1,0.42834214)%
    \put(0.05,0){\includegraphics[scale=1.4]{coxeter.pdf}}%
    \put(0.04806107,0.32){\color[rgb]{0,0,0}\makebox(0,0)[lb]{\smash{  $1$}}}%
    \put(0.105,0.32){\color[rgb]{0,0,0}\makebox(0,0)[lb]{\smash{  $2$}}}%
    \put(0.207,0.32){\color[rgb]{0,0,0}\makebox(0,0)[lb]{\smash{  ${n-1}$}}}%
    \put(0.27,0.32){\color[rgb]{0,0,0}\makebox(0,0)[lb]{\smash{  ${n}$}}}%
    \put(0.08,0.272){\color[rgb]{0,0,0}\makebox(0,0)[lb]{\smash{$4$}}}%
    \put(0.04806107,0.235){\color[rgb]{0,0,0}\makebox(0,0)[lb]{\smash{  $1$}}}%
    \put(0.105,0.235){\color[rgb]{0,0,0}\makebox(0,0)[lb]{\smash{  $2$}}}%
    \put(0.207,0.235){\color[rgb]{0,0,0}\makebox(0,0)[lb]{\smash{  ${n-1}$}}}%
    \put(0.27,0.235){\color[rgb]{0,0,0}\makebox(0,0)[lb]{\smash{  ${n}$}}}%
    \put(0.105,0.145){\color[rgb]{0,0,0}\makebox(0,0)[lb]{\smash{  $3$}}}%
    \put(0.207,0.145){\color[rgb]{0,0,0}\makebox(0,0)[lb]{\smash{  ${n-1}$}}}%
    \put(0.27,0.145){\color[rgb]{0,0,0}\makebox(0,0)[lb]{\smash{  ${n}$}}}%
    \put(0.05562,0.12){\color[rgb]{0,0,0}\makebox(0,0)[lb]{\smash{$2$}}}%
    \put(0.04806107,0.172){\color[rgb]{0,0,0}\makebox(0,0)[lb]{\smash{  $1$}}}%
    \put(0.378,0.32){\color[rgb]{0,0,0}\makebox(0,0)[lb]{\smash{  $2$}}}%
    \put(0.437,0.32){\color[rgb]{0,0,0}\makebox(0,0)[lb]{\smash{  $3$}}}%
    \put(0.549,0.32){\color[rgb]{0,0,0}\makebox(0,0)[lb]{\smash{  ${5}$}}}%
    \put(0.602,0.32){\color[rgb]{0,0,0}\makebox(0,0)[lb]{\smash{  ${6}$}}}%
    \put(0.488,0.32){\color[rgb]{0,0,0}\makebox(0,0)[lb]{\smash{  $4$}}}%
    \put(0.51,0.38){\color[rgb]{0,0,0}\makebox(0,0)[lb]{\smash{  $1$}}}%
    \put(0.378,0.24){\color[rgb]{0,0,0}\makebox(0,0)[lb]{\smash{  $2$}}}%
    \put(0.437,0.24){\color[rgb]{0,0,0}\makebox(0,0)[lb]{\smash{  $3$}}}%
    \put(0.549,0.24){\color[rgb]{0,0,0}\makebox(0,0)[lb]{\smash{  $5$}}}%
    \put(0.602,0.24){\color[rgb]{0,0,0}\makebox(0,0)[lb]{\smash{  $6$}}}%
    \put(0.488,0.24){\color[rgb]{0,0,0}\makebox(0,0)[lb]{\smash{  $4$}}}%
    \put(0.51,0.293){\color[rgb]{0,0,0}\makebox(0,0)[lb]{\smash{  $1$}}}%
    \put(0.653,0.24){\color[rgb]{0,0,0}\makebox(0,0)[lb]{\smash{  ${7}$}}}%
    \put(0.378,0.15){\color[rgb]{0,0,0}\makebox(0,0)[lb]{\smash{  $2$}}}%
    \put(0.437,0.15){\color[rgb]{0,0,0}\makebox(0,0)[lb]{\smash{  $3$}}}%
    \put(0.549,0.15){\color[rgb]{0,0,0}\makebox(0,0)[lb]{\smash{  $5$}}}%
    \put(0.602,0.15){\color[rgb]{0,0,0}\makebox(0,0)[lb]{\smash{  $6$}}}%
    \put(0.488,0.15){\color[rgb]{0,0,0}\makebox(0,0)[lb]{\smash{  $4$}}}%
    \put(0.51,0.205){\color[rgb]{0,0,0}\makebox(0,0)[lb]{\smash{  $1$}}}%
    \put(0.653,0.15){\color[rgb]{0,0,0}\makebox(0,0)[lb]{\smash{  ${7}$}}}%
    \put(0.7,0.15){\color[rgb]{0,0,0}\makebox(0,0)[lb]{\smash{  ${8}$}}}%
    \put(0.04806107,0.05){\color[rgb]{0,0,0}\makebox(0,0)[lb]{\smash{  $1$}}}%
    \put(0.103,0.05){\color[rgb]{0,0,0}\makebox(0,0)[lb]{\smash{  $2$}}}%
    \put(0.159,0.05){\color[rgb]{0,0,0}\makebox(0,0)[lb]{\smash{  $3$}}}%
    \put(0.08,0.087){\color[rgb]{0,0,0}\makebox(0,0)[lb]{\smash{$5$}}}%
    \put(0.04806107,-0.017){\color[rgb]{0,0,0}\makebox(0,0)[lb]{\smash{  $1$}}}%
    \put(0.105,-0.017){\color[rgb]{0,0,0}\makebox(0,0)[lb]{\smash{  $2$}}}%
    \put(0.16,-0.017){\color[rgb]{0,0,0}\makebox(0,0)[lb]{\smash{  $3$}}}%
    \put(0.084,0.0140){\color[rgb]{0,0,0}\makebox(0,0)[lb]{\smash{$5$}}}%
    \put(0.205,-0.017){\color[rgb]{0,0,0}\makebox(0,0)[lb]{\smash{  $4$}}}%
    \put(0.378,0.075){\color[rgb]{0,0,0}\makebox(0,0)[lb]{\smash{  $1$}}}%
    \put(0.437,0.075){\color[rgb]{0,0,0}\makebox(0,0)[lb]{\smash{  $2$}}}%
    \put(0.49,0.075){\color[rgb]{0,0,0}\makebox(0,0)[lb]{\smash{  $3$}}}%
    \put(0.47,0.112){\color[rgb]{0,0,0}\makebox(0,0)[lb]{\smash{$4$}}}%
    \put(0.54,0.075){\color[rgb]{0,0,0}\makebox(0,0)[lb]{\smash{  $4$}}}%
    \put(0.38,-0.005){\color[rgb]{0,0,0}\makebox(0,0)[lb]{\smash{  $1$}}}%
    \put(0.433,-0.005){\color[rgb]{0,0,0}\makebox(0,0)[lb]{\smash{  $2$}}}%
    \put(0.40866966,0.0247){\color[rgb]{0,0,0}\makebox(0,0)[lb]{\smash{$m$}}}%
    \put(0,0.34531973){\color[rgb]{0,0,0}\makebox(0,0)[lb]{\smash{$A_n :$}}}%
    \put( 0,0.26){\color[rgb]{0,0,0}\makebox(0,0)[lb]{\smash{$B_n :$}}}%
    \put( 0,0.1675){\color[rgb]{0,0,0}\makebox(0,0)[lb]{\smash{$D_n :$}}}%
    \put( 0,0.07){\color[rgb]{0,0,0}\makebox(0,0)[lb]{\smash{$H_3 :$}}}%
    \put( 0,0.005){\color[rgb]{0,0,0}\makebox(0,0)[lb]{\smash{$H_4 :$}}}%
    \put(0.32,0.34531973){\color[rgb]{0,0,0}\makebox(0,0)[lb]{\smash{$E_6 :$}}}%
    \put(0.32,0.26){\color[rgb]{0,0,0}\makebox(0,0)[lb]{\smash{$E_7 :$}}}%
    \put(0.32,0.167){\color[rgb]{0,0,0}\makebox(0,0)[lb]{\smash{$E_8 :$}}}%
    \put(0.32,0.093){\color[rgb]{0,0,0}\makebox(0,0)[lb]{\smash{$F_4 :$}}}%
    \put(0.32,0.015){\color[rgb]{0,0,0}\makebox(0,0)[lb]{\smash{$I_{2}(m) :$}}}%
    \put(0.142,0.37){\color[rgb]{0,0,0}\makebox(0,0)[lb]{\smash{  $n\geq 1$}}}%
    \put(0.142,0.28){\color[rgb]{0,0,0}\makebox(0,0)[lb]{\smash{  $n\geq 2$}}}%
    \put(0.142,0.19){\color[rgb]{0,0,0}\makebox(0,0)[lb]{\smash{  $n\geq 4$}}}%
    \put(0.485,0.015){\color[rgb]{0,0,0}\makebox(0,0)[lb]{\smash{  $m\geq 5$}}}%
  \end{picture}%
\endgroup%